\documentclass[11pt,a4paper]{article}
\usepackage{etex} 
\usepackage[latin1]{inputenc}
\usepackage[T1]{fontenc}
\usepackage[right=1cm,left=2cm,top=2cm,bottom=2cm]{geometry}
\usepackage{ifthen}
\usepackage{enumerate}
\usepackage{enumitem}
\usepackage{amsmath}
\usepackage{amsthm}
\usepackage{amsfonts}
\usepackage{amssymb}
\usepackage{latexsym}
\usepackage{ae}
\usepackage{xcolor,colortbl} 
\usepackage{graphics,amsmath,amssymb}
\usepackage[np]{numprint}
\usepackage{hyperref}
\usepackage{multirow}
\usepackage{todonotes}

\npdecimalsign{\ensuremath{.}}

\newtheorem{theorem}{Theorem}

\newtheorem{coro}[theorem]{Corollary}
\newtheorem{lemma}[theorem]{Lemma}
\newtheorem{prop}[theorem]{Proposition}

\newtheorem*{theoremA}{Theorem A}
\newtheorem*{theoremB}{Theorem B}

\theoremstyle{definition}

\theoremstyle{remark}

\numberwithin{equation}{section}

\makeatletter
\newcommand{\subjclass}[2][2020]{%
  \let\@oldtitle\@title%
  \gdef\@title{\@oldtitle\footnotetext{#1 \emph{Mathematics subject classification.} #2}}%
}
\newcommand{\keywords}[1]{%
  \let\@@oldtitle\@title%
  \gdef\@title{\@@oldtitle\footnotetext{\emph{Key words and phrases.} #1.}}%
}
\makeatother

\newcommand{\Z}{\mathbb {Z}}

\newcommand{\R}{\mathbb {R}}

\newcommand{\Addresses}{{
  \bigskip
  \footnotesize

  O.~Bordell\`es, \textsc{2 all\'{e}e de la combe, 43000 Aiguilhe, France.}\par\nopagebreak
  \textit{E-mail address}: \texttt{borde43@wanadoo.fr}
}}

\DeclareMathOperator{\Li}{Li}
\DeclareMathOperator{\li}{li}
\DeclareMathOperator{\sign}{sign}

\DeclareMathOperator{\md}{mod}
\DeclareMathOperator{\un}{I}
\DeclareMathOperator{\de}{II}
\DeclareMathOperator{\tr}{III}
\DeclareMathOperator{\qua}{IV}
\DeclareMathOperator{\id}{id}

\title{{\bf An explicit result for short sums of positive arithmetic functions}}
\date{}

\author{Olivier Bordell\`es}

\subjclass[2020]{Primary 11A25; Secondary 11N37, 11N56.}

\keywords{Arithmetic functions, Mean values, Short intervals, Erd\H{o}s-Hooley $\Delta$ function}

\begin{document}

\maketitle

\begin{abstract}
We prove a totally explicit bound for short sums of certain non-negative arithmetic functions satisfying a general growth condition, and apply this result to derive two explicit estimates for the Erd\H{o}s-Hooley $\Delta$-function in short intervals.
\end{abstract}

\section{Introduction}

\subsection{Background and motivation}

In \cite{shiu80}, P. Shiu proved the following general result dealing with short sums of a large class of multiplicative functions.

\begin{theoremA}
Let $\left( \varepsilon ,\eta, \theta \right) \in \left( 0,\frac{1}{2}\right)^3$, $a < q \in \Z_{\geqslant 1}$ such that $\left( a,q \right) =1$, and $f$ be a non-negative multiplicative function satisfying
\begin{enumerate}
   \item[\scriptsize $\triangleright$] $\exists \,A  >  0$, $f\left( p^\alpha\right) \leqslant A^\alpha$,
   \item[\scriptsize $\triangleright$] $\exists \,B  =  B\left( \eta \right) >0$, $f(n) \leqslant B\,n^{\eta}$,
\end{enumerate}
for all prime numbers $p$ and $n,\alpha \in \Z_{\geqslant 1}$. Then, uniformly for $q < y^{1-\theta}$ and $1 \leqslant x^{\varepsilon} \leqslant y \leqslant x$, we have
$$\sum_{\substack{x-y <n \leqslant x \\ n \equiv a \; (\md q)}} f(n) \ll \frac{y}{\varphi(q) \log x} \exp \left( \sum_{\substack{p\leqslant x \\ p \nmid q}} \frac{f(p)}{p}\right).$$
\end{theoremA}

The idea of the proof goes back to Erd\H{o}s \cite{erd52}, and was next enhanced by Wolke \cite{wol71}. Shiu then applied the Erd\H{o}s-Wolke method to give a uniform bound for the sum over arithmetic progressions, enlarging the family of multiplicative functions $f$, and also giving a short-interval result. This very useful theorem was then generalized by Nair \cite{nair92}, which extended such ideas further to bound sums of the shape $\sum_{x-y < n \leqslant x} f(|P(n)|)$ where $P \in \Z[X]$, and then by Nair-Tenenbaum \cite{nair98} who proved a uniform bound in the multi-dimensional case and relaxing the hypothesis of multiplicativity of $f$. Subsequently, a Nair-Tenenbaum type bound uniform in the discriminant of the polynomial $P$ was given by Henriot \cite{hen12}. Let us also mention an interesting variation of Shiu's result, where the variable runs over some sifted sets, which was given by Pollack \cite{pol20}. In another direction, it is noteworthy to point out the breakthrough paper by Mat\"{o}maki-Radziwi\l\l~\cite{mat16}, in which the authors relate short averages of $\left[ -1, 1 \right]$-valued multiplicative functions to their long averages, by proving that, for any $1 \leqslant y \leqslant X$ such that $y = y(X) \to \infty$ as $X \to \infty$, the formula
$$\frac{1}{y} \sum_{x - y < n \leqslant x} f(n) = \frac{2}{X} \sum_{X/2 < n \leqslant X} f(n) + o(1)$$
holds for all but $o(X)$ integers in the interval $\left[ \frac{1}{2}X,X \right]$. This result has subsequently been enhanced and generalized in several directions, see \cite{mat20,man23,sun24} for instance.

\medskip

As far as we know, there is very few explicit results dealing with short sums of arithmetic functions in the literature. As a matter of fact, the only explicit bound we could find was proved by Landreau in his PhD Thesis \cite{lan87}, a particular version of which can be stated as follows.

\begin{theoremB}
Let $\ell \in \Z_{\geqslant 2}$ be fixed, and let $f : \Z_{\geqslant 1} \to \left[ 1,+ \infty \right)$ be a multiplicative, completely sub-multiplicative, prime-independent arithmetic function, such that the series $\sum_p \ \sum_{\alpha \geqslant 2} \frac{\left( f \left( p^\alpha \right) \right)^\ell}{p^\alpha}$ converges. Then, uniformly for all $x \geqslant 4^\ell$ and $y \geqslant 0$ such that $x^{1/\ell} \leqslant y \leqslant x$, we have
$$\sum_{x - y < n \leqslant x} f(n) \leqslant M_{\ell,f} \, y \, (\log x)^{f(p)^{2 \ell}}$$
where
$$M_{\ell,f} := 2 \ell^{-f(p)^{2 \ell}} f(p)^{\ell^2(\ell-1)} \, \exp \left( f(p)^{2 \ell} + \sum_p f(p)^\ell \sum_{\alpha = 2}^\infty \frac{f \left( p^\alpha \right)^\ell}{p^\alpha}\right).$$
\end{theoremB}

The proof rests on the key role played by certain small divisors. This result is quite useful, but the exponent on the logarithm is almost always too large in the usual applications.  For instance, applied to the Dirichlet-Piltz divisor function $\tau_k$ with $\ell = 2$, this yields
$$\sum_{x - y < n \leqslant x} \tau_k(n) \leqslant  \left( 2 k^4 \left( \tfrac{e}{2} \right)^{k^4} \zeta(2)^{4^k k^3 M_k} \right) y \, (\log x)^{k^4}$$
for $x^{1/2} \leqslant y \leqslant x$, where $M_k := \underset{0 \leqslant h \leqslant 2k-2}{\max} \left( {k-1 \choose h+2}^2 + {2k-1 \choose h+2} + k^2 {2k-1 \choose h+1} \right)$.

\subsection{Main result}

In the light of the above considerations, the aim of this work is to prove a completely explicit bound for short sums of arithmetic functions, not necessarily multiplicative, yielding the true order of magnitude and uniform in a large class of functions.

\begin{theorem}[Main Theorem]
\label{th:main}
Let $\ell \geqslant 1$ and $f$ be a positive arithmetic function such that there exists $k \in \Z_{\geqslant 1}$ such that, for all $(m,n) \in \left( \Z_{\geqslant 1} \right)^2$, we have
\begin{equation}
   f(mn) \leqslant \tau_k(m) f(n). \label{eq:hyp_5}
\end{equation}
Then, uniformly for all $x \geqslant \exp \left( 7^{28} (12e\ell)^{28\log ( 192 e \ell ) } \right)$ and $x^{1/ \ell} \leqslant y \leqslant x$, we have
$$\sum_{x-y < n \leqslant x} f(n) \leqslant \Lambda (k,\ell) \, \frac{y}{\log x} \, \sum_{n \leqslant x} \frac{f(n)}{n}$$
with $\Lambda (k,\ell) := 12 \ell k^{6 \ell} + 5 \ell C_{k,\ell} + 9 \ell D_{k,\ell}$ and
\begin{enumerate}
   \item[\scriptsize $\triangleright$] $C_{k,\ell} := \left( 10 \ell (k-1) \right)^{k^{16 \ell} (k-1)}$,
   \item[\scriptsize $\triangleright$] $D_{k,\ell} := A_k \max \left( A_k, 2k^{5 \ell A_k} \right) + (2A_k)^{-A_k/4} e^{B_{k,\ell}/e} k^{5 \ell B_{k,\ell}} \delta_{\{A_k < B_{k,\ell}\}} + 5$,
   \item[\scriptsize $\triangleright$] $A_k := 125^k \times \frac{\np{1092}}{k^3}$ and $B_{k,\ell} := \frac{1}{2} e^{4/e} k^{20 \ell}$,
   \item[\scriptsize $\triangleright$] $\delta_{\{A_k < B_{k,\ell}\}} = 1$ if $A_k < B_{k,\ell}$ and $0$ otherwise.
\end{enumerate}
\end{theorem}

The proof essentially follows the Erd\H{o}s-Wolke-Shiu method, but we add the following two refinements:
\begin{enumerate}
   \item[\scriptsize $\triangleright$] \cite[Lemma~1]{shiu80} is improved and given in a totally explicit form. See Lemma~\ref{le:shiu_4} and \eqref{eq:shiu_4} ;
   \item[\scriptsize $\triangleright$] Since $f$ is not necessarily multiplicative, \cite[Lemma~4]{shiu80} must be replaced by a completely new result, although still based upon Rankin's method. See Lemma~\ref{le:shiu_9}.
\end{enumerate}

\subsection{Application}

The Hooley $\Delta_k$-function is defined by
\begin{equation}
   \Delta_k (n) := \underset{u_1,\dotsc,u_{k-1} \in \R}{\max} \sum_{\substack{d_1 d_2 \dotsb d_{k-1} \mid n \\ e^{u_i} < d_i \leqslant e^{u_i + 1}}} 1. \label{eq:Hooley_fct}
\end{equation}
By convention, $\Delta_1 = \mathbf{1}$, and it is customary to set $\Delta_2 := \Delta$ and to call it the Erd\H{o}s-Hooley $\Delta$-function. The average order of $\Delta_k$ was first investigated by Hooley \cite{hoo79}, and then improved by Hall-Tenenbaum \cite{hall88}. When $k=2$, Hall-Tenenbaum's result was first improved by de la Bret\`{e}che-Tenenbaum \cite{breten22}, and then sharpened by Koukoulopoulos-Tao in \cite{koutao23}, which in turn was improved very recently in \cite{breten23} by de la Bret\`{e}che \& Tenenbaum again, who proved that, for large real $x$,
$$x (\log \log x)^{3/2} \ll \sum_{n \leqslant x} \Delta(n) \ll x (\log \log x)^{5/2}.$$
It seems to be very tricky to derive explicit upper bounds from the method of Hall-Tenenbaum, Koukoulopoulos-Tao or de la Bret\`{e}che-Tenenbaum, each one of them relying on induction processes dealing with properties of the moments
$$M_q(n) := \int_{- \infty}^{+ \infty} \Delta(n;u)^q \, \textrm{d}u$$
where $\Delta(n;u)$ is defined in \eqref{eq:Delta(n;u)} below. In contrast, using Hooley's technique allows us to get explicit results with quite reasonable implied constants.

It is well-known by \cite[Lemma~61.1]{hall88} that the function $\Delta_k$ satisfies \eqref{eq:hyp_5}. Hence applying the Main Theorem~\ref{th:main} yields the following corollaries.

\begin{coro}
\label{cor:Hooley}
Let $\ell \geqslant 1$. Uniformly for all $x \geqslant \exp \left( 7^{28} (12e\ell)^{28\log ( 192 e \ell ) } \right)$ and $x^{1/ \ell} \leqslant y \leqslant x$, we have
$$\sum_{x-y < n \leqslant x} \Delta(n) \leqslant \phi (\ell) \, y (\log x)^{-1+4/\pi}$$
where $\phi(\ell) := \np{16748} \Lambda (2,\ell)$.
\end{coro}

\begin{coro}
\label{cor:Hooley_restricted_omega}
Let $\ell \geqslant 1$ and $j \in \Z_{\geqslant 1}$. Uniformly for all
$$x > \max \left( \exp \left( 7^{28} (12e\ell)^{28\log ( 192 e \ell ) } \right) \, , \, \exp \left( e^{\pi j/4} \right) \right)$$
and $x^{1/ \ell} \leqslant y \leqslant x$, we have
$$\sum_{\substack{x-y < n \leqslant x \\ \omega(n) \leqslant j}} \Delta(n) \leqslant \phi (\ell) \, \frac{y}{\log x} \left( \frac{4e \log \log x}{\pi j} \right)^j$$
where $\phi(\ell) := \np{16748} \Lambda (2,\ell)$.
\end{coro}

\subsection{Notation}

In what follows, $f : \Z_{\geqslant 1} \to \R_{\geqslant 0}$ is an arithmetic function that always satisfies \eqref{eq:hyp_5}. $p$ always denotes a prime number, $j,k,m,n,r,s \in \Z_{\geqslant 1}$, $\ell,q \in \R_{\geqslant 1}$, $c_0:=\np{1.25506}$, $c \in \left[ 0,1 \right]$ and $\mathfrak{m} := \frac{2}{\pi} \left( \sqrt{\pi^2-4} - 2 \arccos \frac{2}{\pi} \right)$. Let $\log_k$ denote the $k$-fold iterated logarithm. The Hooley $\Delta_k$-function is defined in \eqref{eq:Hooley_fct}, and hence, for $k=2$, the Erd\H{o}s-Hooley $\Delta$-function is given by $\Delta(n) := \displaystyle \max_{u \in \R} \Delta(n;u)$ where, for all $u \in \R$,
\begin{equation}
   \Delta(n;u):= \sum_{\substack{d \mid n \\ e^{u} < d \leqslant e^{u + 1}}} 1. \label{eq:Delta(n;u)}
\end{equation}
The constants of the Main Theorem~\ref{th:main} will always be used in the whole text:
\begin{align}
   & C_{k,\ell} := \left( 10 \ell (k-1) \right)^{k^{16 \ell} (k-1)}, \label{eq:C_{k,l}} \\
   & D_{k,\ell} := A_k \max \left( A_k, 2k^{5 \ell A_k} \right) + (2A_k)^{-\frac{A_k}{4}} e^{\frac{1}{e}B_{k,\ell}} k^{5 \ell B_{k,\ell}}\delta_{\{A_k < B_{k,\ell}\}} + 5, \label{eq:D_{k,l}} \\
   & A_k := 125^k \times \tfrac{\np{1092}}{k^3}, \label{eq:A_k} \\
   & B_{k,\ell} := \tfrac{1}{2} e^{4/e} k^{20 \ell}. \label{eq:B_{k,l}} 
\end{align}
Finally, the notation $f(x) = M(x) + O^\star (R(x))$ means $\left| f(x) - M(x) \right | \leqslant R(x)$.

\section{Proof of Theorem~\ref{th:main}}

\subsection{Tools}

\subsubsection{Basics}

\begin{lemma}
\label{le:ineg_log}
Let $a,b > 0$ such that $ab \geqslant \frac{1}{4}e^2$ and set $c := \frac{e}{e-1}$. Then
$$x \geqslant  \left( 2^{\log 16} (ab)^{\log ( 16 a b ) } \right)^{c^2a}  \Longrightarrow \log x \geqslant a \left( \log (b \log x) \right)^2.$$
\end{lemma}

\begin{proof}
Since $ab \geqslant \frac{1}{4}e^2$ and $x \geqslant e^{e^2/b}$,
$$\log x \geqslant a \left( \log (b \log x) \right)^2 \iff x \geqslant \exp \left(4 a W_{-1}^2 \left(- \tfrac{1}{2 \sqrt{ab}} \right) \right),$$
where $W_{-1}$ is the second real-valued branch of the Lambert $W$-function. The result then follows by applying the lower bound $W_{-1} (x) \geqslant c \log(-x)$ for $-\frac{1}{e} \leqslant x < 0$, which can be found in \cite[Theorem~3.1]{alz18}.
\end{proof}

\begin{lemma}
\label{le:ineg_r}
Let $r,k \in \Z_{\geqslant 1}$. Set $c_1:=\frac{63 \times 2^{2/3}}{20}$ and $A_k$ given in \eqref{eq:A_k}. Then
$$r \geqslant A_k \Longrightarrow c_1 5^{k}k r^{2/3} \leqslant \tfrac{r}{12}  (\log r)^{2}.$$ 
\end{lemma}

\begin{proof}
Set $c_2 := 12 c_1$ and $c_k := 5^{k}k$. We have to prove that
\begin{equation}   
   r \geqslant A_k \Longrightarrow r (\log r)^{6} \geqslant(c_2c_k)^3 . \label{eq:ineq_r}
\end{equation}
But, since $r \geqslant 1$ and $c_2c_k \geqslant e$
$$r (\log r)^{6} \geqslant(c_2c_k)^3  \iff r \geqslant (c_2c_k)^3 \left( 6 W_0 \left( \tfrac{1}{6} \sqrt{c_2c_k}\right) \right)^{-6}$$
where $W_0$ is the first real-valued branch of the Lambert $W$-function. It is known \cite{hoor08} that, for $t \geqslant e$, $W_0(t) \geqslant \log t - \log \log t \geqslant \frac{1}{2} \log t$, so that
$$r \geqslant (c_2c_k)^3 \left( 3 \log \left( \tfrac{1}{6} \sqrt{c_2c_k} \right) \right)^{-6} \Longrightarrow r (\log r)^{6} \geqslant(c_2c_k)^3 $$
and noticing that $\sqrt{c_2} > 6$ and $(2/3)^6 c_2^3 < \np{18967}$, we derive
$$r \geqslant \np{18967} c_k^3 (\log c_k)^{-6} \Longrightarrow r (\log r)^{6} \geqslant(c_2c_k)^3.$$
Now \eqref{eq:ineq_r} follows with the bound
$$\np{18967} c_k^3 (\log c_k)^{-6} = \np{18967} \dfrac{k^3 \, 125^k}{(k \log 5 + \log k)^6} < A_k. \eqno\qedhere$$
\end{proof}

\begin{lemma}
\label{le:ineg_tau_k}
Let $k \in \Z_{\geqslant 2}$ be fixed. Then, for all $n \in \Z_{\geqslant 1}$ and all $0 < \varepsilon \leqslant \frac{k-1}{6}$,
$$\tau_k(n) \leqslant \min \left( k^{\Omega(n)} \, , \, \left( \tfrac{3(k-1)}{5\varepsilon}\right)^{k^{1/\varepsilon}(k-1)} \, n^\varepsilon  \right).$$
\end{lemma}

\begin{proof}
The inequality $\tau_k(n) \leqslant k^{\Omega(n)}$ is well-known and also valid for $k=1$, see for example \cite[(9)]{san89}. As for the $2$nd inequality, for $k \geqslant 2$, we have
\begin{align*}
   \frac{\tau_k(n)}{n^\varepsilon} &= \prod_{p^\nu \| n} \frac{\tau_k \left( p^\nu \right)}{p^{\nu \varepsilon}} = \Biggl( \prod_{\substack{p^\nu \| n \\ p < k^{1/\varepsilon}}} \times \prod_{\substack{p^\nu \| n \\ p \geqslant k^{1/\varepsilon}}} \Biggr) {k + \nu - 1 \choose \nu} \frac{1}{p^{\nu \varepsilon}} \\
   & \leqslant \prod_{\substack{p^\nu \| n \\ p < k^{1/\varepsilon}}}  \frac{(\nu+1)^{k-1}}{p^{\nu \varepsilon}} \times \prod_{\substack{p^\nu \| n \\ p \geqslant k^{1/\varepsilon}}} \underbrace{\frac{k^\nu}{p^{\nu \varepsilon}}}_{\leqslant 1} \leqslant \prod_{\substack{p^\nu \| n \\ p < k^{1/\varepsilon}}} \frac{(\nu+1)^{k-1}}{2^{\nu \varepsilon}}
\end{align*}
and the function $\nu \longmapsto (\nu+1)^{k-1} 2^{-\nu \varepsilon}$ reaches on $\left[ 1,+ \infty \right)$ a maximum in $\nu = \frac{k-1}{\varepsilon \log 2} - 1$ with value $2^\varepsilon \left( \frac{k-1}{\varepsilon \, e \log 2}\right)^{k-1} \leqslant \left( \frac{3(k-1)}{5\varepsilon}\right)^{k-1}$ since $\varepsilon \leqslant \frac{k-1}{6}$, completing the proof.
\end{proof}

\begin{lemma}
\label{le:factorial}
For all $N \in \Z_{\geqslant 1}$
$$\prod_{n=1}^N n! \leqslant N^{\frac{1}{2}(N+1)^2} \, e^{-\frac{3}{4}(N^2-1)}.$$
\end{lemma}

\noindent
\textit{Proof.} Let $P_N$ be the product of the left-hand side. Then
$$\log P_N = \sum_{n=1}^N \; \sum_{h=1}^n \log h = \sum_{h=1}^{N} \log h \sum_{n=h}^N 1 = \sum_{h=1}^{N} (N-h+1)\log h$$
and the Euler-Maclaurin summation formula yields
\begin{align*}
   \log P_N &= \int_1^N (N-x+1) \log x \, \textrm{d}x + \tfrac{1}{2} \log N + \tfrac{1}{12} \left( \tfrac{1}{N} - \log N - N \right) \\
   & \hspace*{1cm} + \tfrac{1}{2} \int_1^N \frac{N+x+1}{x^2} B_2 \left( \left\lbrace x \right\rbrace \right) \, \textrm{d}x.
\end{align*}
Now since $-\frac{1}{12} \leqslant B_2 \left( \left\lbrace x \right\rbrace \right) \leqslant \frac{1}{6}$, we get
\begin{align*} 
   \log P_N & \leqslant \int_1^N (N-x+1) \log x \, \textrm{d}x + \tfrac{1}{2} \log N + \tfrac{1}{12} \left( \tfrac{1}{N} - \log N - N \right) \\
   & \hspace*{1cm} + \tfrac{1}{12} \int_1^N \frac{N+x+1}{x^2} \, \textrm{d}x = \tfrac{1}{2} (N+1)^2 \log N - \tfrac{3}{4} \left( N^2 - 1 \right).\tag*{\qed}
\end{align*}

\begin{lemma}
\label{le:shiu_3}
Uniformly for all $\alpha \geqslant 3$ and $x > 1$,
$$\sum_{p \leqslant x} \frac{1}{(\log p)^\alpha} < (2 \alpha)^\alpha \, \frac{x}{(\log x)^{\alpha+1}}.$$
\end{lemma}

\begin{proof}
Let $b > 1$ a parameter at our disposal. Using \cite[(3.6)]{ros62}, we derive
\begin{align*}
   \sum_{p \leqslant x} \frac{1}{(\log p)^\alpha} &= \left( \sum_{p \leqslant x^{1/b}} + \sum_{x^{1/b} < p \leqslant x} \right) \frac{1}{(\log p)^\alpha} \\
   & \leqslant \frac{\pi \left( x^{1/b} \right) }{(\log 2)^\alpha} + \frac{b^\alpha \pi(x)}{(\log x)^\alpha} \\
    & \leqslant c_0 \left( \frac{b x^{1/b}}{(\log 2)^\alpha \log x} + \frac{ b^\alpha x}{(\log x)^{\alpha + 1}} \right)
\end{align*}
where $c_0:=\np{1.25506}$, and the inequality
$$x^{1/b} \leqslant \left( \frac{b \alpha}{e (b-1)}\right)^{\alpha} \frac{x}{(\log x)^{\alpha}}$$
valid for all $x > 1$, $b > 1$ and $\alpha > 0$, along with the obvious bound $b^\alpha \leqslant b^{\alpha+1}$ in the $2$nd term, yields
$$\sum_{p \leqslant x} \frac{1}{(\log p)^\alpha} < \frac{c_0 b^{\alpha+1} x}{(\log x)^{\alpha + 1}} \left( \left( \frac{\alpha}{e (b-1) \log 2}\right)^{\alpha} + 1 \right).$$
Now choose $b = 1 + \left( \frac{\alpha}{e \log 2}\right)^{\frac{\alpha}{\alpha+1}}$, so that
$$\sum_{p \leqslant x} \frac{1}{(\log p)^\alpha} < \frac{c_0 x}{(\log x)^{\alpha + 1}} \left(\left( \frac{\alpha}{e \log 2}\right)^{\frac{\alpha}{\alpha+1}} + 1 \right)^{\alpha+2}$$
and it is not difficult to see that $c_0 \left(\left( \frac{\alpha}{e \log 2}\right)^{\frac{\alpha}{\alpha+1}} + 1 \right)^{\alpha+2} < (2 \alpha)^\alpha$ when $\alpha \geqslant 3$.
\end{proof}

\begin{lemma}
\label{le:1/p_alpha}
Let $0 \leqslant \alpha < 1$. For all $x \geqslant 2$
$$\sum_{p \leqslant x} \frac{1}{p^\alpha} < \frac{c_0 (1 + 2\alpha)}{1-\alpha} \, \frac{x^{1-\alpha}}{\log x}$$
where $c_0:=\np{1.25506}$.
\end{lemma}

\begin{proof}
By partial summation and \cite[(3.6)]{ros62} as above, we get
\begin{align*}
   \sum_{p \leqslant x} \frac{1}{p^\alpha} &= \frac{\pi(x)}{x^\alpha} + \alpha \int_2^x \frac{\pi(t)}{t^{\alpha+1}} \, \textrm{d}t \leqslant c_0 \left( \frac{x^{1-\alpha}}{\log x} + \alpha \int_2^x \frac{\textrm{d}t}{t^\alpha \log t} \right) \\
   & \leqslant c_0 \left( \frac{x^{1-\alpha}}{\log x} + \frac{3 \alpha x^{1-\alpha}}{(1-\alpha) \log x} \right) = \frac{c_0 (1 + 2\alpha)}{1-\alpha} \, \frac{x^{1-\alpha}}{\log x}
\end{align*}
as required.
\end{proof}

\begin{lemma}
\label{le:tech}
Let $a < b \in \R$ and $g$ be a positive, derivable function on $\left[ a,b \right]$, such that $g^{\, \prime}$ is non-decreasing and satisfies
$$\exists \, \lambda_1 > 0, \ \forall x \in \left[ a,b \right], \ g^{\, \prime}(x) \geqslant \lambda_1.$$
Then
$$\int_a^b e^{-g(x)} \, \mathrm{d}x \leqslant \frac{e^{-g(a)}}{\lambda_1}.$$
\end{lemma}

\begin{proof}
This is \cite[lemma~4.6]{bortoth}.
\end{proof}

\begin{lemma}
\label{le:shiu_10}
Let $k \in \Z_{\geqslant 1}$, $\ell \in \R_{\geqslant 1}$ and $A \in \R_{\geqslant 2}$. Then
$$\sum_{r=2}^A r k^{5 \ell r} < A \max \left( A, 2k^{5 \ell A}\right).$$
\end{lemma}

\noindent
\textit{Proof.}
The sum is $< A^{2}$ if $k=1$, and otherwise
\begin{align*}
   \sum_{r=2}^A r k^{5 \ell r} &= \frac{k^{5 \ell(A+1)} (A k^{5 \ell} - A - 1) - k^{10 \ell} (k^{5 \ell} - 2)}{(k^{5 \ell} - 1)^2} \\
   & \leqslant A k^{5 \ell A} \left( \frac{k^{5 \ell}}{k^{5 \ell}-1}\right)^{2} < 2 A k^{5 \ell A}. \tag*{\qed}
\end{align*}

\begin{lemma}
\label{le:shiu_2}
If $2 \leqslant y \leqslant x$ and $2 \leqslant z \leqslant \sqrt{y}$,
$$\sum_{\substack{x-y < n \leqslant x \\ P^-(n) > z}} 1 \leqslant \frac{2y}{\log z}.$$
\end{lemma}

\begin{proof}
This is a straightforward application of the large sieve. More precisely, we use \cite[Theorem~9.8]{friwa10}, along with the usual notation of the large sieve, applied to the sequence $a_n = 1$ if $P^- (n) > z$ and $0$ otherwise. Note that, for all primes $p \leqslant z$, we have $\omega(p) \geqslant 1$ since $0 \in \Omega_p$ in this case. Let $Q \in \left[ 1,z \right]$ be any real parameter. Then \cite[Theorem~9.8]{friwa10}, with $M = \lfloor x - y \rfloor$ and $N = \lfloor x \rfloor - \lfloor x-y \rfloor$, yields
$$\sum_{\substack{x-y < n \leqslant x \\ P^- (n) > z}} 1 \leqslant \frac{N-1+Q^2}{\mathcal{L}} \leqslant \frac{y+Q^2}{\mathcal{L}}$$
where $\displaystyle \mathcal{L} := \sum_{k \leqslant Q} \mu(k)^2 \prod_{p \mid k} \frac{\omega(p)}{p-\omega(p)}$. Since $Q \leqslant z$, each prime divisor $p$ of $k \leqslant Q \leqslant z$ satisfies $\omega(p) \geqslant 1$, so that
$$\displaystyle \mathcal{L} \geqslant \sum_{k \leqslant Q} \mu(k)^2 \prod_{p \mid k} \frac{1}{p-1} = \sum_{k \leqslant Q} \frac{\mu^2(k)}{\varphi(k)} \geqslant \log Q$$
and hence
$$\sum_{\substack{x-y < n \leqslant x \\ P^- (n) > z}} 1 \leqslant \frac{y+Q^2}{\log Q}$$
for all $1 < Q \leqslant z$. Choosing $Q=z$, so that $Q^2 \leqslant y$, yields the desired bound.
\end{proof}

\subsubsection{Key tools}

\begin{lemma}
\label{le:shiu_4}
Assume
\begin{equation}
   x \geqslant \exp \left( 7^{28} (12e\ell)^{28\log ( 192 e \ell )} \right) \label{eq:condition_shui_4}
\end{equation}
and $z \geqslant x^{\frac{1}{3 \ell}}$. Then
$$\sum_{\substack{n \leqslant z \\ P^+(n) < \log x \log_2 x}} 1 < z^{1/4}.$$
\end{lemma}

\begin{proof}
For all $2 \leqslant t \leqslant z$, set as usual
\begin{equation}
   \Psi (z, t) := \sum_{\substack{n \leqslant z \\ P^+(n) \leqslant t}} 1 \label{eq:Psi}
\end{equation}
and let $a > 0$ be a parameter to optimize. By Rankin's method, we derive
$$\Psi(z,t) \leqslant z^a \prod_{p \leqslant t} \left( 1 + \frac{1}{p^a-1}\right) \leqslant \exp \left( a \log z + \sum_{p \leqslant t} \frac{1}{p^a-1} \right).$$
The \textsc{am-gm} inequality implies that, for all $N \in \Z_{\geqslant 1}$
$$p^a-1 \geqslant \sum_{n=1}^N \frac{(a \log p)^n}{n!} \geqslant N \left( \prod_{n=1}^N \frac{1}{n!} \right)^{1/N} \left( a \log p \right)^{(N+1)/2}$$
so that
$$\Psi(z,t) \leqslant \exp \left( a \log z + \frac{K_N}{a^{(N+1)/2}} \sum_{p \leqslant t} \frac{1}{(\log p)^{(N+1)/2}} \right)$$
with $K_N := N^{-1} \left( \prod_{n=1}^N n! \right)^{1/N}$. Lemma~\ref{le:shiu_3} implies that, if $N \geqslant 5$
$$\Psi(z,t) \leqslant \exp \left( a \log z + \frac{L_N t}{a^{(N+1)/2} (\log t)^{(N+3)/2}} \right)$$
where $L_N := (N+1)^{\frac{N+1}{2}} \times K_N$. The choice $a = \frac{1}{\log t} \left( \frac{(N+1)L_N t}{2\log z} \right)^{\frac{2}{N+3}}$ then gives
$$\Psi (z, t) \leqslant \exp \left( \frac{R_N t^{\frac{2}{N+3}} (\log z)^{\frac{N+1}{N+3}}}{\log t} \right)$$
where $R_N := (N+3)(N+1)^{-\frac{N+1}{N+3}} \left( \frac{1}{2} L_N \right)^{\frac{2}{N+3}}$. We have
$$R_N = (N+3) (2N)^{-\frac{2}{N+3}} \times \left( \prod_{n=1}^N n! \right)^{\frac{2}{N(N+3)}}   < N \times \left( \prod_{n=1}^N n! \right)^{\frac{2}{N(N+3)}}$$
since $N \geqslant 5$, and by Lemma~\ref{le:factorial}, we also have
$$\left( \prod_{n=1}^N n! \right)^{\frac{2}{N(N+3)}} \leqslant N^{1 - \frac{N-1}{N(N+3)}} e^{- \frac{3(N^2-1)}{2N(N+3)}} < e^{-1} N$$
if $N \geqslant 5$. Hence, $R_N < e^{-1} N^2$ if $N \geqslant 5$. Substituting $t = \log x \log _2 x$ implies that, for all $N \in \Z_{\geqslant 5}$,
$$\sum_{\substack{n \leqslant z \\ P^+(n) < \log x \log_2 x}} 1 < \exp \left( \frac{N^2}{e} \frac{(\log x)^{\frac{2}{N+3}} (\log z)^{\frac{N+1}{N+3}}}{(\log_2 x)^{\frac{N+1}{N+3}}}\right)$$
and since $x \leqslant z^{3 \ell}$, we derive
$$\sum_{\substack{n \leqslant z \\ P^+(n) < \log x \log_2 x}} 1 < \exp \left( \frac{N^2 (3 \ell)^{\frac{2}{N+3}}}{e}  \frac{\log z}{(\log_2 x)^{\frac{N+1}{N+3}}}\right).$$
The function $u \mapsto u^2 (3 \ell)^{\frac{2}{u+3}} (\log_2 x)^{-\frac{u+1}{u+3}}$ is increasing on the interval
$$\Bigl[ \log( 3 \ell \log_2 x) - 3 \, , \, + \infty \Bigr)$$
if $x \geqslant \exp \left( e^{\frac{1}{3 \ell} e^{12}} \right)$. Choosing $N = \left \lfloor \log( 3 \ell \log_2 x) \right \rfloor$ then yields
\begin{equation}
   \sum_{\substack{n \leqslant z \\ P^+(n) < \log x \log_2 x}} 1 < \exp \left( \frac{e\log z \left( \log( 3 \ell \log_2 x) \right)^2}{\log_2 x} \right). \label{eq:shiu_4}
\end{equation}
Finally, using Lemma~\ref{le:ineg_log} with $x$ replaced by $\log x$ and $a = 4e$ and $b=3\ell$, we see that the hypothesis \eqref{eq:condition_shui_4} implies
$$e \left( \log( 3 \ell \log_2 x) \right)^2 \leqslant \tfrac{1}{4} \log_2 x$$
as required. We conclude the proof by noticing that the condition \eqref{eq:condition_shui_4} also implies the following three claims:
\begin{enumerate}
   \item[\scriptsize $\triangleright$] Clearly, $x \geqslant \exp \left( e^{\frac{1}{3 \ell} e^{12}} \right)$ .
   \item[\scriptsize $\triangleright$] Next, we have $\log( 3 \ell \log_2 x) > 5$, so that $N \geqslant 5$.
   \item[\scriptsize $\triangleright$] Finally, note that the inequality $x^{\frac{1}{3 \ell}} \geqslant \log x \log_2 x$ is ensured as soon as $\log x \geqslant 5 \ell \log_2 x$, the latter condition being satisfied for $x \geqslant (10 \ell \log 5 \ell)^{5 \ell}$, a constraint easily fullfilled by \eqref{eq:condition_shui_4}. Hence $z \geqslant x^{\frac{1}{3 \ell}} \geqslant t$.
\end{enumerate}
The proof is complete.
\end{proof}

\begin{lemma}
\label{le:shiu_9}
Let $z \geqslant e^{e^2}$ and $k,r \in \Z_{\geqslant 1}$. Assume $r \geqslant A_k$, $\log z \geqslant r \log 2r$ and that $f$ satisfies \eqref{eq:hyp_5}. Then
$$\sum_{\substack{z^{1/2} < n \leqslant z \\ P^+(n) \leqslant z^{1/r}}} \frac{f(n)}{n} <  e^{-\frac{1}{4} r \log 2r} \sum_{n \leqslant z} \frac{f(n)}{n}.$$
\end{lemma}

\begin{proof}
Set $t:=z^{1/2}$ and $u:=z^{1/r}$, and let $0 < a \leqslant \frac{2}{3}$ be a parameter at our disposal. If $J_a = \id_a \star \, \mu$ is the Jordan arithmetical function of parameter $a$, Rankin's method and \eqref{eq:hyp_5} yield
\begin{align*}
   \sum_{\substack{t < n \leqslant z \\ P^+(n) \leqslant u}} \frac{f(n)}{n} & \leqslant t^{-a} \sum_{\substack{t < n \leqslant z \\ P^+(n) \leqslant u}} \frac{f(n)}{n} n^a = t^{-a} \sum_{\substack{t < n \leqslant z \\ P^+(n) \leqslant u}} \frac{f(n)}{n} \sum_{d \mid n} J_a(d) \\
   & \leqslant t^{-a} \sum_{P^+(d) \leqslant u} \frac{J_a(d)}{d} \sum_{\substack{m \leqslant z/d \\ P^+(m) \leqslant u}} \frac{f(md)}{k} \\
   & \leqslant t^{-a} \sum_{P^+(d) \leqslant u} \frac{J_a(d) \tau_k(d)}{d} \sum_{m \leqslant z} \frac{f(m)}{m}
\end{align*}
with
\begin{align*}
   \sum_{P^+(d) \leqslant u} \frac{J_a(d) \tau_k(d)}{d} & = \prod_{p \leqslant u} \left( 1 + \left( 1 - \frac{1}{p^a} \right) \sum_{\alpha = 1}^\infty \frac{{k+\alpha-1 \choose \alpha}}{p^{(1-a)\alpha}} \right)  \\
   & = \prod_{p \leqslant u} \left( 1 + \left( 1 - \frac{1}{p^a} \right) \left( \left( 1 - \frac{1}{p^{1-a}} \right)^{-k} - 1 \right)  \right) \\
   & \leqslant \exp \left( \sum_{p \leqslant u} \left( \left( 1 - \frac{1}{p^{1-a}} \right)^{-k} - 1 \right) \right)
\end{align*}
and the inequality $(1-x)^{-k}-1 \leqslant 5^{k}kx$, valid for $0 \leqslant x \leqslant \tfrac{4}{5}$ and $k \geqslant 1$, implies that
$$\sum_{P^+(d) \leqslant u} \frac{J_a(d) \tau_k(d)}{d} \leqslant \exp \left( 5^{k} k \sum_{p \leqslant u} \frac{1}{p^{1-a}} \right).$$
Using Lemma~\ref{le:1/p_alpha} with $\alpha = 1 - a$, we derive
\begin{align*}
   \sum_{P^+(d) \leqslant u} \frac{J_a(d) \tau_k(d)}{d} & \leqslant \exp \left( \frac{5^{k} \times \np{1.26} k (3-2a)}{a} \frac{u^a}{\log u} \right) \\ 
   & < \exp \left( \frac{21 \times 5^{k} r k}{10a} \frac{z^{a/r}}{\log z} \right).
\end{align*}
Inserting this estimate in the calculations above, we get
$$\sum_{\substack{t < n \leqslant z \\ P^+(n) \leqslant u}} \frac{f(n)}{n} < \exp \left( - \frac{a}{2} \log z  + \frac{21 \times 5^{k} r k}{10a} \frac{z^{a/r}}{\log z} \right) \sum_{m \leqslant z} \frac{f(m)}{m}$$
and the choice of $a= \frac{2r \log 2r}{3\log z}$ yields 
$$\sum_{\substack{t < n \leqslant z \\ P^+(n) \leqslant u}} \frac{f(n)}{n} < \exp \left( - \frac{r}{3} \log 2r + \frac{63 \times 2^{2/3} \times 5^{k} k r^{2/3}}{20 \log 2r} \right) \sum_{m \leqslant z} \frac{f(m)}{m}$$
and since $r \geqslant A_k$, Lemma~\ref{le:ineg_r} yields 
$$\sum_{\substack{t < n \leqslant z \\ P^+(n) \leqslant u}} \frac{f(n)}{n} < \exp \left( - \frac{r}{3} \log 2r + \frac{r}{12} \log 2r \right) \sum_{m \leqslant z} \frac{f(m)}{m}$$
as asserted.
\end{proof}

\subsection{Proof of Theorem~\ref{th:main}}

In what follows, we set
\begin{equation}
   z := y^{1/3}. \label{eq:z}
\end{equation}
Write each integer $n \in \left( x-y,x \right]$ as $n = p_1^{\alpha_1} \dotsb p_J^{\alpha_J} p_{J+1}^{\alpha_{J+1}} \dotsb p_r^{\alpha_r} := a_n b_n$, with $a_n \leqslant z < a_n p_{J+1}^{\alpha_{J+1}}$ if such an integer $J$ exists, $a_n = 1$ otherwise. Subdivise these integers into four classes:
\begin{enumerate}
   \item[\scriptsize $\triangleright$] \begin{footnotesize}\textsc{Class I}.\end{footnotesize} $P^-(b_n) \geqslant z^{1/2}$ ;
   \item[\scriptsize $\triangleright$] \begin{footnotesize}\textsc{Class II}.\end{footnotesize} $P^-(b_n) < z^{1/2}$ and $a_n \leqslant z^{1/2}$ ;
   \item[\scriptsize $\triangleright$] \begin{footnotesize}\textsc{Class III}.\end{footnotesize} $P^-(b_n) \leqslant \log x \log_2 x$ and $a_n > z^{1/2}$ ;
   \item[\scriptsize $\triangleright$] \begin{footnotesize}\textsc{Class IV}.\end{footnotesize} $\log x \log_2 x < P^-(b_n) < z^{1/2}$ and $a_n > z^{1/2}$.
\end{enumerate}

\subsubsection{Class $\un$}
By \eqref{eq:hyp_5} and Lemma~\ref{le:ineg_tau_k}, we have
$$\sum_{n \in \, \un} f(n) \leqslant \sum_{a \leqslant z} f(a) \, \sum_{\substack{\frac{x-y}{a} <  b \leqslant \frac{x}{a} \\ P^-(b) > z^{1/2}}} k^{\Omega(b)}$$
and since $P^-(b) > z^{1/2} = y^{1/6} \geqslant x^{\frac{1}{6 \ell}}$, we get $\Omega(b) \leqslant \frac{\log x}{\log P^-(b)}\leqslant 6 \ell$, and hence, using Lemma~\ref{le:shiu_2}, we get
\begin{align}
   \sum_{n \in \, \un} f(n) & \leqslant k^{6 \ell} \sum_{a \leqslant z} f(a) \sum_{\substack{\frac{x-y}{a} <  b \leqslant \frac{x}{a} \\ P^-(b) > z^{1/2}}} 1 \leqslant 4k^{6 \ell} \, \frac{y}{\log z} \sum_{a \leqslant z} \frac{f(a)}{a} \notag \\ 
   &  \leqslant 12 \ell k^{6\ell} \, \frac{y}{\log x} \sum_{n \leqslant x} \frac{f(n)}{n}. \label{eq:somme_I_modif}
\end{align}

\subsubsection{Class $\de$}
Noticing that to each $n \in \de$, there correspond a prime number $p$ and an integer $s$ such that $p^s \| n$, $p \leqslant z^{1/2}$ and $p^s > z^{1/2}$, and also noticing that $s_p$ is the least integer $s$ such that $p^s > z^{1/2}$, we derive 
$$\sum_{n \in \, \de} 1  \leqslant \sum_{p \leqslant z^{1/2}} \; \sum_{\substack{x-y < n \leqslant x \\ p^{s_p} \mid n}} 1 = \sum_{p \leqslant z^{1/2}} \left( \left \lfloor \frac{x}{p^{s_p}} \right \rfloor - \left \lfloor \frac{x-y}{p^{s_p}} \right \rfloor \right) \leqslant \sum_{p \leqslant z^{1/2}} \left( \frac{y}{p^{s_p}} + 1 \right).$$
Now $p \leqslant z^{1/2}$ and $p^{s_p} > z^{1/2}$ imply that $s_p \geqslant 2$, and hence $p^{-s_p} \leqslant \min \left( z^{-1/2}, p^{-2} \right)$, so that, using $z \geqslant 8^4$
\begin{align*}
   \sum_{n \in \, \de} 1 & \leqslant y \left( \sum_{p \leqslant z^{1/4}} \frac{1}{z^{1/2}} + \sum_{p > z^{1/4}} \frac{1}{p^2} \right) + z^{1/2} \\
   & \leqslant y z^{-1/4} \left( 1 + \frac{\np{2.002}}{\log(z^{1/4})} \right) + z^{1/2} \\
   & < 2 y z^{-1/4} + y^{1/4} \leqslant 3 y^{11/12}. 
\end{align*}
Hypothesis \eqref{eq:hyp_5} and Lemma~\ref{le:ineg_tau_k}, applied with $\varepsilon = (16 \ell)^{-1}$, imply that
$$f(n) \leqslant C_{k,\ell} \, n^{\frac{1}{16 \ell}} \leqslant C_{k,\ell} \, x^{\frac{1}{16 \ell}} \leqslant C_{k,\ell} \, y^{1/16}$$
where $C_{k,\ell}$ is given in \eqref{eq:C_{k,l}}, so that
\begin{equation}
   \sum_{n \in \, \de} f(n) \leqslant 3 A \, C_{k,\ell} y^{47/48} \leqslant 3 C_{k,\ell} \frac{y}{\log y} \leqslant 3 \ell C_{k,\ell} \frac{y}{\log x}. \label{eq:somme_II_modif}
\end{equation}

\subsubsection{Class $\tr$}
Using the notation \eqref{eq:Psi}, we have
\begin{align*}
   \sum_{n \in \, \tr} 1 & \leqslant \sum_{\substack{z^{1/2} < a \leqslant z \\ P^+(a) < \log x \log_2 x}} \; \sum_{\substack{x-y < n \leqslant x \\ a \mid n}} 1 \leqslant \sum_{\substack{z^{1/2} < a \leqslant z \\ P^+(a) < \log x \log_2 x}} \left( \frac{y}{a} + 1 \right) \\
   & \leqslant y z^{-1/2} \Psi(z, \log x \log_2 x) + z,
\end{align*}
and Lemma~\ref{le:shiu_4} yields
$$\sum_{n \in \, \tr} 1 \leqslant y z^{-1/4} + z \leqslant 2 y^{11/12}.$$
Furthermore, as in the case of sum II, we have $f(n) \leqslant C_{k,\ell} \, y^{1/16}$, and therefore
\begin{equation}
   \sum_{n \in \, \tr} f(n) \leqslant 2 \ell C_{k,\ell} \frac{y}{\log x}. \label{eq:somme_III_modif}
\end{equation}

\subsubsection{Class $\qua$}
We have
$$\sum_{n \in \, \qua} f(n) \leqslant \sum_{z^{1/2} < a \leqslant z} f(a) \ \sum_{\substack{x-y < n \leqslant x \\ a_n = a \\ \max(P^+(a), \log x \log_2 x) < P^-(b_n) \leqslant z^{1/2}}} k^{\Omega(b_n)}.$$
Set $R := \left \lfloor \frac{\log z}{\log \left( \log x \log_2 x \right)} \right \rfloor$, so that $z^{\frac{1}{R+1}} < \log x \log _2 x$, and, for all $r \in \{2, \dotsc,R \}$, consider the integers $n$ satisfying $z^{\frac{1}{r+1}} < P^-(b_n) \leqslant z^{1/r}$. For those integers, we have $P^+(a_n) = P^+(a) < P^-(b_n) \leqslant z^{1/r}$ and also
$$\Omega(b_n) \leqslant \frac{\log x}{\log P^-(b_n)} < \frac{(r+1) \log x}{\log z} \leqslant 3 (r+1) \ell \leqslant 5 r \ell.$$
We deduce that
\begin{align*}
   \sum_{n \in \, \qua} f(n) & \leqslant \sum_{r=2}^{R} k^{5 r \ell} \; \sum_{\substack{z^{1/2} < a \leqslant z \\ P^+(a) < z^{1/r}}} f(a) \; \sum_{\substack{x-y < n \leqslant x \\ a \mid n \\ z^{\frac{1}{r+1}} < P^-(n/a) \leqslant z^{1/r}}} 1 \\
   & \leqslant \sum_{r=2}^{R} k^{5 r \ell} \; \sum_{\substack{z^{1/2} < a \leqslant z \\ P^+(a) < z^{1/r}}} f(a) \; \sum_{\substack{\frac{x-y}{a} < d \leqslant \frac{x}{a} \\ P^-(d) > z^{\frac{1}{r+1}}}} 1 \\
   & \leqslant \frac{2 y}{\log z} \, \sum_{r=2}^{R} (r+1) k^{5 r \ell} \; \sum_{\substack{z^{1/2} < a \leqslant z \\ P^+(a) < z^{1/r}}} \frac{f(a)}{a} \\
   & \leqslant \frac{3 y}{\log z} \, \sum_{r=2}^{R} r k^{5 r \ell} \; \sum_{\substack{z^{1/2} < a \leqslant z \\ P^+(a) < z^{1/r}}} \frac{f(a)}{a} \\
   & = \frac{3 y}{\log z} \, \left( \sum_{r=2}^{A_k} + \sum_{A_k < r \leqslant R} \right) r k^{5 r \ell} \; \sum_{\substack{z^{1/2} < a \leqslant z \\ P^+(a) < z^{1/r}}} \frac{f(a)}{a} \\
   & := \Sigma_1 + \Sigma_2,
\end{align*}
say, where we used Lemma~\ref{le:shiu_2} in the $3$rd line, and the bound $r+1 \leqslant \frac{3}{2}r$ in the $4$th one, and where $A_k$ is given in \eqref{eq:A_k}. The sum $\Sigma_1$ is treated trivially with the help of Lemma~\ref{le:shiu_10} yielding
\begin{equation}
   \Sigma_1 < 9 A_k \ell \max \left( A_k, 2k^{5 \ell A_k}\right) \frac{y}{\log x} \sum_{n \leqslant z} \frac{f(n)}{n}. \label{eq:Sigma_1}
\end{equation}
As for $\Sigma_2$, note that $A_k < r \leqslant R \leqslant \frac{\log z}{\log_2 x} \leqslant \frac{\log z}{\log_2 z}$ and $\log z \geqslant \frac{1}{3 \ell} \log x \geqslant e^2$. Hence
$$r \log 2r \leqslant \log z \times \frac{\log \left( \frac{2\log z}{\log_2 z} \right) }{\log_2 z} \leqslant \log z$$
so that Lemma~\ref{le:shiu_9} applies, and gives
\begin{align*}
   \Sigma_2 & < \frac{3 y}{\log z} \, \sum_{A_k < r \leqslant R} \exp \left( -\tfrac{r}{4} \log 2r + 5 r \ell \log k + \log r \right) \sum_{n \leqslant z} \frac{f(n)}{n} \\\\
   & \leqslant \frac{3y}{\log z} \ \sum_{r > A_k} \exp \left( -\tfrac{r}{4} \log 2r +  r \phi_{k,\ell} \right) \sum_{n \leqslant z} \frac{f(n)}{n}
\end{align*}
where $\phi_{k,\ell} := 5 \ell \log k + e^{-1}$, and where we used $\log r \leqslant e^{-1}r$. Now set $g(r) := \tfrac{r}{4} \log 2r -  r \phi_{k,\ell}$ and recall $B_{k,\ell}$ which is given in \eqref{eq:B_{k,l}}.
\begin{enumerate}
   \item[\scriptsize $\triangleright$] If $A_k \geqslant B_{k,\ell}$, then $r > B_{k,\ell}$ and hence $g(r) > 0$, $g^{\, \prime}$ is strictly increasing and $g^{\, \prime}(r) \geqslant \frac{1}{4}$, so that, by Lemma~\ref{le:tech}, we get $\sum_{r > A_k} e^{-g(r)} \leqslant 5$, and thus
   \begin{equation}
      \Sigma_2 \leqslant 45 \ell \frac{y}{\log x} \sum_{n \leqslant z} \frac{f(n)}{n}. \label{eq:Sigma_2_1}
   \end{equation}
   \item[\scriptsize $\triangleright$] If $A_k < B_{k,\ell}$, then
$$\sum_{A_k < r \leqslant B_{k,\ell}} e^{-g(r)} \leqslant (2A_k)^{-A_k/4} \ \sum_{A_k < r \leqslant B_{k,\ell}} e^{r \phi_{k,\ell}} \leqslant  (2A_k)^{-A_k/4}e^{B_{k,\ell} \phi_{k,\ell}}.$$
   In this case, we then have
   \begin{equation}
      \Sigma_2 <9 \ell \left( (2A_k)^{-A_k/4}e^{B_{k,\ell} \phi_{k,\ell}} + 5 \right) \, \frac{y}{\log x} \sum_{n \leqslant z} \frac{f(n)}{n}. \label{eq:Sigma_2_2}
   \end{equation}
\end{enumerate} 
Therefore, by \eqref{eq:Sigma_1} and \eqref{eq:Sigma_2_1} or \eqref{eq:Sigma_2_2}, we get
\begin{equation}
   \sum_{n \in \, \qua} f(n) < 9 \ell D_{k,\ell} \,\frac{y}{\log x} \sum_{n \leqslant z} \frac{f(n)}{n} \label{eq:somme_IV_modif}
\end{equation}
with $D_{k,\ell}$ defined in \eqref{eq:D_{k,l}}.

\subsubsection{Conclusion}

\noindent
The result follows with \eqref{eq:somme_I_modif}, \eqref{eq:somme_II_modif}, \eqref{eq:somme_III_modif} and \eqref{eq:somme_IV_modif}.
\qed

\section{Proof of Corollaries~\ref{cor:Hooley} and~\ref{cor:Hooley_restricted_omega}}

\begin{prop}
\label{pro:Delta_bis}
For all $x \geqslant 1$ and all $t \in \left[ 0,1\right] $
$$\sum_{n \leqslant x} t^{\omega(n)} \Delta (n) < \np{9380} \, x ( \log ex)^{-1+4t/\pi}$$
and
$$\sum_{n \leqslant x} \frac{t^{\omega(n)}\Delta (n)}{n} < \np{16748}  \, ( \log ex)^{4t/\pi}.$$
\end{prop}

\noindent
The $2$nd inequality comes from the first one and a partial summation. Therefore, the rest of the text is devoted to the proof of the bound
\begin{equation}
   \sum_{n \leqslant x} t^{\omega(n)}\Delta (n) < \np{9380} \, x ( \log ex)^{-1+4t/\pi}. \label{eq:explicit_Hooley}
\end{equation}

Hooley's method rests on the use of the following new arithmetic function: for all $n \in \Z_{\geqslant 1}$ and $v \in \R$, define
$$\tau(n;v) := \sum_{d \mid n} d^{iv}.$$
Note that this function is multiplicative and satisfies the trivial bound
\begin{equation}
    \left| \tau(n;v) \right| \leqslant \tau(n). \label{eq:trivial_tau}
\end{equation}

\subsection{Tools}

\begin{lemma}
\label{le:mean}
Let $f : \Z_{\geqslant 1} \to \R_{\geqslant 0}$ be a non-negative multiplicative function, such that there exists $a_1,a_2,a_3 > 0$ such that, for all $x \geqslant 2$, we have
\begin{center}
\begin{tabular}{llll}
(i) & $\displaystyle \frac{1}{x} \, \sum_{p \leqslant x} f(p) \log p \leqslant a_1$; & (ii) & $\displaystyle \sum_p \sum_{\alpha \geqslant 2} \frac{f \left( p^\alpha \right)}{p^{\alpha}} \leqslant a_2$; \\
(iii) & $\displaystyle \sum_p \sum_{\alpha \geqslant 2} \frac{f \left( p^\alpha \right) \log p^{\alpha}}{p^{\alpha}} \leqslant a_3$. &
\end{tabular}
\end{center}
Then, for all $x\geqslant 1$, we have
$$\sum_{n \leqslant x} f(n) \leqslant e^{a_2} \left( a_1+a_3+1\right) \frac{x}{\log(ex)} \exp \left( \sum_{p\leqslant x} \frac{f(p)}{p} \right).$$
\end{lemma}

\begin{proof}
See \cite[Theorem~01]{hall88} or \cite[Theorem~4.9]{bor20}.
\end{proof}

\begin{lemma}
\label{le:TNP_explicite}
For all $x \geqslant 2$
$$\pi(x) = \li(x) + O^\star \left( \frac{\np{4.6} \, x}{(\log x)^3} \right) \quad \left( x \geqslant 2 \right)$$
and 
$$\sum_{p \leqslant x} \frac{1}{p} = \log \log x + B + O^\star \left( \frac{\np{0.6}}{(\log x)^2} \right) \quad \left( x \geqslant 127 \right).$$
\end{lemma}

\begin{proof}
If $p_{1}=2 \leqslant x \leqslant p_{384} = \np{2657}$, we numerically check the first inequality via
$$\left| \frac{(\log x)^3}{x} \left( \pi(x) - \li(x) \right) \right| \leqslant \left| \frac{(\log p_n)^3}{p_n} \left( n - \li(p_n) \right) \right|$$
for all $x \in \left[ p_n \, , \, p_{n+1} \right]$ and all $n \in \{1, \dotsc,384\}$. If $\np{2657} < x \leqslant \np{2.169} \times 10^{25}$, then we use \cite[(3.2)]{john22} stating that, in this region, we have $\left| \pi(x) - \li(x) \right| < \frac{1}{8 \pi} \, \sqrt{x} \, \log x \leqslant \frac{\np{4.6}x}{(\log x)^3}$. Finally, if $x > \np{2.169} \times 10^{25}$, we make appeal to \cite[Theorem~9.1]{saou15} stating that
$\left| \pi(x) - \li(x) \right| < \frac{\np{0.51}x}{(\log x)^3}$. The second inequality can be proved in a similar way: if $p_{31}=127 \leqslant x \leqslant p_{62} = 293$, we numerically check it via
$$\left| (\log x)^2 \left( \sum_{p \leqslant x} \frac{1}{p} - \log \log x - B \right) \right| \leqslant \left| (\log p_n)^2 \left( \sum_{p \leqslant p_n} \frac{1}{p} - \log \log p_n - B \right) \right|$$
for all $x \in \left[ p_n \, , \, p_{n+1} \right]$ and all $n \in \{31, \dotsc,62\}$, and if $x > 293$, then we use \cite[Theorem~5]{ros62} giving $\sum_{p \leqslant x} \frac{1}{p} = \log \log x + B + O^\star \left( \frac{\np{0.5}}{(\log x)^2} \right)$.
\end{proof}

\begin{lemma}
\label{le:moyenne_periodic_fct}
Let $f : \R \to \R$ be a $\mathbb{T}$-periodic function, $\mathbb{T} \geqslant 1$. Set 
\begin{enumerate}
   \item[(i)] $\displaystyle \mu_f := \frac{1}{\mathbb{T}} \int_0^\mathbb{T} f(t) \, \mathrm{d}t$ ;
   \item[(ii)] $\displaystyle m_f := \max_{0 \leqslant c < 1} \left| \int_0^{c \mathbb{T}} \left( f(t) - \mu_f \right) \, \mathrm{d}t \right|$.
\end{enumerate}
Then, for all $1 \leqslant a \leqslant b$, we have
$$\int_a^b \frac{f(t)}{t} \, \mathrm{d}t = \mu_f \log \left( \frac{b}{a} \right) + O^\star \left( \frac{2m_f}{a} \right).$$ 
\end{lemma}

\begin{proof}
Set $g(t) := f(t) - \mu_f$, $\displaystyle L(a,b) := \int_a^b \frac{g(t)}{t} \, \mathrm{d}t$ and, for all $x \geqslant 1$, set $\displaystyle G(x) := \int_0^x g(t) \, \mathrm{d}t$. For all $x \in \R_{\geqslant 1}$ and all $n \in \Z$, we have
$$G \left( x+n \mathbb{T} \right) = G(x) + n G(\mathbb{T}) = G(x)$$
since $G(\mathbb{T}) = 0$. In particular with $n = - \lfloor x/\mathbb{T} \rfloor$, we get
$$G(x) = G \left( \left\lbrace \tfrac{x}{\mathbb{T}} \right\rbrace \mathbb{T} \right)$$
where $\{t\}$ is the fractional part of $t \in \R$. Integrating by parts, we derive
\begin{align*}
   L(a,b) &= \int_a^b \frac{1}{t} \, \textrm{d}G(t) = \frac{G(b)}{b} - \frac{G(a)}{a} + \int_a^b \frac{G(t)}{t^2} \, \mathrm{d}t \\
   & = \frac{1}{b} G \left( \left\lbrace \tfrac{b}{\mathbb{T}} \right\rbrace \mathbb{T} \right) -  \frac{1}{a} G \left( \left\lbrace \tfrac{a}{\mathbb{T}} \right\rbrace \mathbb{T} \right) + \int_a^b  G \left( \left\lbrace \tfrac{t}{\mathbb{T}} \right\rbrace \mathbb{T} \right) \, \frac{\mathrm{d}t}{t^2} 
\end{align*}
and therefore
$$\left| L(a,b) \right| \leqslant \max_{a \leqslant x \leqslant b} \left| G \left( \left\lbrace \tfrac{x}{\mathbb{T}} \right\rbrace \mathbb{T} \right) \right| \left( \frac{1}{b} + \frac{1}{a} + \int_a^b \frac{\textrm{d}t}{t^2} \right) = \frac{2m_f}{a}.$$
The proof follows with $\displaystyle \int_a^b \frac{f(t)}{t} \, \mathrm{d}t = \mu_f \int_a^b \frac{\mathrm{d}t}{t} + L(a,b)$.
\end{proof}

\begin{lemma}
\label{le:min}
For all $c \in \left[ 0,1 \right]$,
$$\left| \int_0^{2c \pi} \left( \left| \cos \tfrac{t}{2} \right| - \tfrac{2}{\pi} \right) \, \mathrm{d}t \right| \leqslant \tfrac{2}{\pi} \left( \sqrt{\pi^2-4} - 2 \arccos \tfrac{2}{\pi} \right).$$
\end{lemma}

\begin{proof}
Straightforwardly, for all $c \in \left[ 0,1 \right]$, we have
$$\int_0^{2c \pi} \left( \left| \cos \tfrac{t}{2} \right| - \tfrac{2}{\pi} \right) \, \textrm{d}t = \begin{cases} 2 \sign \left( \cos c \pi \right) \sin c \pi - 4c, & \textrm{if} \ 0 \leqslant c < \frac{1}{2} \, ; \\ & \\ 0, & \textrm{if} \ c = \frac{1}{2} \, ; \\ & \\ 2 \sign \left( \cos c \pi \right) \sin c \pi + 4(1-c) , & \textrm{if} \ \frac{1}{2} < c \leqslant 1 \end{cases}$$
and the result follows by observing that the absolute value of function on the right-hand side reaches a maximum, with value $\tfrac{2}{\pi} \left( \sqrt{\pi^2-4} - 2 \arccos \tfrac{2}{\pi} \right)$, at the points $\frac{1}{\pi} \arccos \left(  \pm \frac{2}{\pi} \right)$.
\end{proof}

\subsection{A Mertens type bound for $\left| \tau(n;v) \right|$}

\begin{lemma}
\label{le:tau_restreint_bis}
Set $B \approx \np{0.26149 72128 47643} \dotsc$ the Meissel-Mertens constant. Assume that there exist $a>0$, $r \in \Z_{\geqslant 2}$ and $T_1 \geqslant 2$ such that, for all $t \geqslant T_1$, we have
\begin{equation}
   \left| \pi(t) - \li(t) \right| \leqslant \frac{at}{(\log t)^r} \label{eq:H_1}
\end{equation}
and there exist $b>0$, $s \in \Z_{\geqslant 1}$ and $T_2 \geqslant 2$ such that, for all $t \geqslant T_2$, we have
\begin{equation}
\sum_{p \leqslant t} \frac{1}{p} \leqslant \log \log t + B + \frac{b}{(\log t)^s}. \label{eq:H_2}
\end{equation}
Set $\mathfrak{m} := \frac{2}{\pi} \left( \sqrt{\pi^2-4} - 2 \arccos \frac{2}{\pi} \right) $. Then, for all $q \geqslant 1$, all $\max \left(e^q, T_1,T_2\right) \leqslant T \leqslant x$ and all $v \in \left[ 0,1\right]$, we have
\begin{align*}
   \sum_{p \leqslant x} \frac{\left| \tau(p;v) \right|}{p} & < 2B + \frac{4a \zeta(r)}{(2 \pi)^r} + \frac{14a}{(\log T)^r} + \frac{2b}{(\log T)^s} \\
   &  + \begin{cases} 2 \log \log x , & \textrm{if} \ 0 \leqslant v \leqslant \frac{q}{\log x} \, ; \\ & \\ \frac{4}{\pi} \log \log x - \left( 2 - \frac{4}{\pi} \right) \log \left( \frac{v}{q}\right)  + 4\mathfrak{m}q^{-1}, & \textrm{if} \ \frac{q}{\log x} \leqslant v \leqslant \frac{q}{\log T} \, ; \\ & \\ \frac{4}{\pi} \log \log x +\frac{4}{\pi}  \log \left( \frac{v}{q}\right) + 2 \log \log T + 4\mathfrak{m}q^{-1}, & \textrm{if} \ \frac{q}{\log T} \leqslant v \leqslant 1. \end{cases}
\end{align*}
\end{lemma}

\begin{proof}
If $v = 0$, then by \eqref{eq:H_2}
$$\sum_{p \leqslant x} \frac{\left| \tau(p;0) \right|}{p} = 2 \sum_{p \leqslant x} \frac{1}{p} \leqslant 2 \log \log x + 2B + \frac{2b}{(\log x)^s} < 2 \log \log x + 2B + \frac{2b}{(\log T)^s}$$
so that we may suppose $v > 0$. 
\item[]
Set $F(x) := \left| 1 + e^{ix} \right|$ and $R(x) := \pi(x) - \li(x)$. First by \eqref{eq:H_2}, we have
\begin{align}
   \sum_{p \leqslant x} \frac{\left| \tau(p;v) \right|}{p} &= \left( \sum_{p \leqslant T} + \sum_{T < p \leqslant x} \right) \frac{\left| \tau(p;v) \right|}{p} \notag \\
   & \leqslant 2 \sum_{p \leqslant T} \frac{1}{p} + \sum_{T < p \leqslant x} \frac{\left| \tau(p;v) \right|}{p} \notag \\
   & \leqslant 2 \log \log T + 2B + \frac{2b}{(\log T)^s} + \sum_{T < p \leqslant x} \frac{\left| \tau(p;v) \right|}{p}. \label{eq:preliminaries}
\end{align}
By partial summation
\begin{align*}   
   \sum_{T < p \leqslant x} \frac{\left| \tau(p;v) \right|}{p} &= \int_{T}^x \frac{F(v \log t)}{t} \, \textrm{d} \pi (t) \\
   &= \int_T^x \frac{F(v \log t)}{t \log t} \, \textrm{d}t + \int_{T^+}^x \frac{F(v \log t)}{t} \, \textrm{d} R(t) := I_1 + I_2.
\end{align*}
Let us first estimate $I_2$. Integrating by parts yields
\begin{align*}
   I_2 &= \frac{R(x)}{x} F(v \log x) - \frac{R(T)}{T} F(v \log T) - \int_T^x R(t) \, \textrm{d} \left( \frac{F(v \log t)}{t} \right) \\
   &= \frac{R(x)}{x} F(v \log x) - \frac{R(T)}{T} F(v \log T) \\
   & \hspace*{1cm} - \int_T^x \frac{R(t)}{t} \, \textrm{d} \left( F(v \log t) \right) + \int_T^x R(t) F(v \log t) \frac{\textrm{d}t}{t^2}.
\end{align*}
By \eqref{eq:H_1}, we respectively have:
\begin{enumerate}
   \item[]
   \item[\scriptsize $\triangleright$] $\displaystyle \left| \frac{R(x)}{x} F(v \log x) - \frac{R(T)}{T} F(v \log T) \right| \leqslant 2 \left( \frac{\left| R(x) \right|}{x} + \frac{\left| R(T) \right|}{T} \right) \leqslant \frac{4a}{(\log T)^r}$.
   \item[]
   \item[\scriptsize $\triangleright$] $\displaystyle \left| \int_T^x R(t) F(v \log t) \frac{\textrm{d}t}{t^2}\right| \leqslant 2 \int_T^x \frac{\left| R(t) \right|}{t^2} \, \textrm{d}t \leqslant 2a \int_T^\infty \frac{\textrm{d}t}{t (\log t)^r} = \frac{2a}{(\log T)^r}.$
   \item[]
   \item[\scriptsize $\triangleright$] For the $3$rd error term, set $H := \left \lfloor \frac{v}{2 \pi} \log \frac{x}{T} \right \rfloor$. By periodicity of $F$, we have
   \begin{align*}
      & \int_T^x \frac{R(t)}{t} \, \textrm{d} \left( F(v \log t) \right) \\
      & \hspace*{1cm} = \int_{v \log T}^{v \log x} R \left( e^{u/v} \right) e^{-u/v} \, \textrm{d} F(u) \\
      &\hspace*{1cm} = \sum_{h=0}^{H-1} \int_{v \log T + 2h \pi}^{v \log T + 2(h+1) \pi} R \left( e^{u/v} \right) e^{-u/v} \, \textrm{d} F(u) \\
      & \hspace*{2cm} + \int_{v \log T + 2H \pi}^{v \log x} R \left( e^{u/v} \right) e^{-u/v} \, \textrm{d} F(u) \\
      & \hspace*{1cm} = \sum_{h=0}^{H-1} \int_{v \log T}^{v \log T + 2 \pi} R \left( e^{(w+2h \pi)/v} \right) e^{-(w+2h \pi)/v} \, \textrm{d} F(w + 2 h \pi) \\
      & \hspace*{2cm} + \int_{v \log T + 2H \pi}^{v \log x} R \left( e^{u/v} \right) e^{-u/v} \, \textrm{d} F(u)\\
      & \hspace*{1cm} = \int_{v \log T}^{v \log T + 2 \pi}  \left( \sum_{h=0}^{H-1} R \left( e^{(w+2h \pi)/v} \right) e^{-(w+2h \pi)/v} \right) \textrm{d}F(w) \\
      & \hspace*{2cm} + \int_{v \log T + 2H \pi}^{v \log x} R \left( e^{u/v} \right) e^{-u/v} \, \textrm{d} F(u).
   \end{align*}
   Noticing that $v \log T + 2H \pi \in \left( v \log x - 2 \pi \, , \, v \log x \right]$, we have by \eqref{eq:H_1}
   \begin{align*}
      & \left| \int_{v \log T + 2H \pi}^{v \log x} R \left( e^{u/v} \right) e^{-u/v} \, \textrm{d} F(u) \right| \\
      & \hspace*{1cm} \leqslant \sup_{v \log T + 2H \pi \leqslant u \leqslant v \log x} \left( \left| R \left( e^{u/v} \right) \right| e^{-u/v} \right) \times V_{v \log T + 2H \pi}^{v \log x} \left( F \right) \\
      & \hspace*{1.5cm} \leqslant a \, \sup_{v \log T \leqslant u \leqslant v \log x} \left( \frac{v^r}{u^r} \right) \times V_{0}^{2 \pi} \left( 2 \left| \cos \tfrac{x}{2} \right| \right) \\
      & \hspace*{2cm} = \frac{2a}{(\log T)^r} \times  V_0^{2 \pi} \left( \left| \cos \tfrac{x}{2} \right| \right) = \frac{4a}{(\log T)^r}
   \end{align*}
   since $V_0^{2 \pi} \left( \left| \cos \tfrac{x}{2} \right| \right) = 2$, and, still using \eqref{eq:H_1} and setting $\zeta(s,a)$ the Hurwitz zeta function, we get
   \begin{align*}
      & \left| \int_{v \log T}^{v \log T + 2 \pi} \textrm{d}F(w) \, \sum_{h=0}^{H-1} R \left( e^{(w+2h \pi)/v} \right) e^{-(w+2h \pi)/v} \right| \\
      & \hspace*{3cm} \leqslant a \, \int_{v \log T}^{v \log T + 2 \pi} \sum_{h=0}^\infty \left( \frac{w+2h \pi}{v} \right)^{-r} \left| \textrm{d}F(w) \right| \\
      & \hspace*{3cm} = \frac{a v^r}{(2 \pi)^r}  \int_{v \log T}^{v \log T + 2 \pi} \zeta \left( r, \tfrac{w}{2 \pi}\right)  \left| \textrm{d}F(w) \right| \\
      & \hspace*{3cm} < \frac{a v^r}{(2 \pi)^r} \int_{v \log T}^{v \log T + 2 \pi} \left( \frac{(2 \pi)^r}{w^r} + \zeta(r) \right) \left| \textrm{d}F(w) \right| \\
      & \hspace*{3cm} \leqslant \frac{a v^r}{(2 \pi)^r} \left( \frac{(2 \pi)^r}{v^r (\log T)^r} + \zeta(r) \right) \int_{v \log T}^{v \log T + 2 \pi} \left| \textrm{d}F(w) \right| \\
      & \hspace*{3cm} = \frac{2a v^r}{(2 \pi)^r} \left( \frac{(2 \pi)^r}{v^r (\log T)^r} + \zeta(r) \right) V_0^{2 \pi} \left( \left| \cos \tfrac{x}{2} \right| \right) \\
      & \hspace*{3cm} \leqslant \frac{4a}{(\log T)^r} + \frac{4a \zeta(r)}{(2 \pi)^r} 
   \end{align*}
   where we used $v \leqslant 1$ in the ultimate line.
\end{enumerate}  
Thus
\begin{equation}
   \left| I_2 \right| < \frac{14a}{(\log T)^r} + \frac{4a \zeta(r)}{(2 \pi)^r}. \label{eq:I_2_bis}
\end{equation}
Now let us estimate $I_1$. If $0 < v \leqslant \frac{q}{\log x}$, the trivial bound $F(v \log t) \leqslant 2$ gives
\begin{equation}
   I_1 \leqslant 2 \int_T^x \frac{\textrm{d}t}{t \log t} = 2 \log \log x - 2 \log \log T. \label{eq:I_1_first_bis}
\end{equation}
If $\frac{q}{\log x} \leqslant v \leqslant \frac{q}{\log T}$, then
\begin{align*}
   I_1 & = \left( \int_T^{e^{q/v}} + \int_{e^{q/v}}^x \right) \frac{F(v \log t)}{t \log t} \, \textrm{d}t \\
   & \leqslant 2 \int_T^{e^{q/v}} \frac{\textrm{d}t}{t \log t} + \int_{e^{q/v}}^x \frac{F(v \log t)}{t \log t} \, \textrm{d}t \\
   & = - 2 \log \left( q^{-1} v \log T \right) + \int_q^{v \log x} \frac{F(u)}{u} \, \textrm{d}u \\
   & = 2 \left( -  \log \left( q^{-1} v \log T \right)  + \int_q^{v \log x} \frac{\left| \cos \frac{u}{2} \right|}{u} \, \textrm{d}u \right) \\
   &= 2 \left( \tfrac{2}{\pi} \log \left( q^{-1} v \log x  \right) -  \log \left( q^{-1} v \log T  \right) + O^{\star} ( 2\mathfrak{m}q^{-1} ) \right)
\end{align*}
where we used Lemma~\ref{le:moyenne_periodic_fct} with $a=q$, $b=v \log x$, $f(u) = \left| \cos \frac{u}{2} \right|$, $\mathbb{T} = 2 \pi$, $\mu_f = \frac{2}{\pi}$ and, by Lemma~\ref{le:min}, $m_f = \mathfrak{m}$ in the ultimate line. Thus, in this case,
\begin{equation}
   I_1 \leqslant 2 \left( \tfrac{2}{\pi} \log \log x - \log \log T + \left(\tfrac{2}{\pi} - 1 \right) \log\left( \tfrac{q}{v}\right) + 2\mathfrak{m}q^{-1} \right). \label{eq:I_1_second_bis}
\end{equation}
If $\frac{q}{\log T} \leqslant v \leqslant 1$, then
$$I_1 =  \int_{e^{q/v}}^x  \frac{F(v \log t)}{t \log t} \, \textrm{d}t - \int_{e^{q/v}}^T  \frac{F(v \log t)}{t \log t} \, \textrm{d}t \leqslant \int_{e^{q/v}}^x  \frac{F(v \log t)}{t \log t} \, \textrm{d}t$$
and the above calculations yield in this case
\begin{equation}
   I_1 \leqslant 2 \left( \tfrac{2}{\pi} \log \log x + \tfrac{2}{\pi} \log\left( \tfrac{q}{v}\right) + 2\mathfrak{m}q^{-1} \right). \label{eq:I_1_second_ter}
\end{equation}
The result then follows from \eqref{eq:preliminaries}, \eqref{eq:I_2_bis}, \eqref{eq:I_1_first_bis}, \eqref{eq:I_1_second_bis} and \eqref{eq:I_1_second_ter}.
\end{proof}

\subsection{Proof of \eqref{eq:explicit_Hooley}}

\subsubsection{Estimates for certain sums of primes}

\begin{lemma}
\label{le:sums_primes}
\begin{enumerate}
   \item[]
   \item[\scriptsize $\triangleright$] For all $x \geqslant 2$ and all $0 \leqslant v \leqslant 1$
   $$\frac{1}{x} \sum_{p \leqslant x} | \tau(p;v) | \log p \leqslant 2 + \np{3.88} \times 10^{-8}.$$
   \item[\scriptsize $\triangleright$] For all $0 \leqslant v \leqslant 1$
   $$\sum_p \sum_{\alpha \geqslant 2} \frac{\left | \tau\left( p^\alpha;v \right) \right|}{p^{\alpha}} < \np{2.9215}.$$
   \item[\scriptsize $\triangleright$] For all $0 \leqslant v \leqslant 1$  
   $$\sum_p \sum_{\alpha \geqslant 2} \frac{\left | \tau\left( p^\alpha;v \right) \right|}{p^{\alpha}} \log p^\alpha < \np{8.16001}.$$
\end{enumerate}
\end{lemma}

The proof uses the next bound.

\begin{lemma}
\label{le:Ramare}
Let $f \in C^1 \left[ 2,\infty \right) $, $f \geqslant 0$, $f^{\, \prime} < 0$ and such that $f(t) = o(1/t)$ as $t \to \infty$. Set $c_1 := 1 + \np{1.94} \times 10^{-8}$. Then, for all $x \geqslant 2$
$$\sum_{p > x} f(p) \log p \leqslant c_1 \int_x^\infty f(t) \, \mathrm{d}t + (c_1-1)x f(x) + \frac{4xf(x)}{(\log x)^2}.$$
\end{lemma}

\begin{proof}
This is \cite[Lemma~3.2]{rama}, used with $\epsilon = c_1-1$ instead of $\epsilon = 1/\np{36260}$, since it is known that $\theta(x) < c_1 x$ for all $x \geqslant 0$ by \cite[Corollary~2.1]{bro21}.
\end{proof}

\begin{proof}[Proof of Lemma~\ref{le:sums_primes}]
We use \eqref{eq:trivial_tau}.
\begin{enumerate}
   \item[\scriptsize $\triangleright$] For all $x \geqslant 2$ and all $0 \leqslant v \leqslant 1$
   $$\frac{1}{x} \sum_{p \leqslant x} | \tau(p;v) | \log p \leqslant \frac{2 \theta(x)}{x} \leqslant 2c_1 = 2 + \np{3.88} \times 10^{-8}.$$
   \item[\scriptsize $\triangleright$] For all $0 \leqslant v \leqslant 1$
   $$\sum_p \sum_{\alpha \geqslant 2} \frac{\left | \tau\left( p^\alpha;v \right) \right|}{p^{\alpha}} \leqslant \sum_p \sum_{\alpha \geqslant 2} \frac{\alpha + 1}{p^\alpha} = \sum_p \frac{3p-2}{p(p-1)^2}.$$
   Lemma~\ref{le:Ramare} gives $\displaystyle \sum_{p > 10^4} \frac{3p-2}{p(p-1)^2} < \np{3.12} \times 10^{-5}$. Using \textsc{pari/gp}, we get $\displaystyle \sum_{p \leqslant 10^4} \frac{3p-2}{p(p-1)^2} < \np{2.92135}$, completing the proof of the second bound.
   \item[\scriptsize $\triangleright$] For all $0 \leqslant v \leqslant 1$  
   $$\sum_p \sum_{\alpha \geqslant 2} \frac{\left | \tau\left( p^\alpha;v \right) \right|}{p^{\alpha}} \log p^\alpha \leqslant \sum_p \log p \sum_{\alpha \geqslant 2} \frac{\alpha(\alpha + 1)}{p^\alpha} = 2\sum_p \frac{3p^2 - 3p + 1}{p(p-1)^3} \log p.$$
   As above, Lemma~\ref{le:Ramare} yields $\displaystyle \sum_{p > 10^6} \frac{2(3p^2 - 3p + 1)}{p(p-1)^3} \log p < \np{6.126} \times 10^{-6}$. With \textsc{pari/gp}, $\displaystyle \sum_{p \leqslant 10^6} \frac{2(3p^2 - 3p + 1)}{p(p-1)^3} \log p < \np{8.159998}$. This completes the proof.
\end{enumerate}
\end{proof}

\subsubsection{Finalization of the proof of \eqref{eq:explicit_Hooley}}

One may suppose $x \geqslant 4 \times 10^{41}$, since, otherwise, the trivial bound
\begin{align*}
   \sum_{n \leqslant x} t^{\omega(n)}\Delta(n) & \leqslant \sum_{n \leqslant x} \tau(n) \leqslant x \, \log(ex) \\
   & < \np{9380} \, x \,  (\log ex)^{-1} \leqslant \np{9380} \, x \,  (\log ex)^{-1+4t/\pi}
\end{align*}
is sufficient. For $z \in \R$ such that $|z| \leqslant 1$, set
$$\kappa(z) := \int_{-1}^1 e^{izv} \, \textrm{d}v = \frac{2\sin z}{z}.$$
Since $|z| \leqslant 1$, we have
$$\kappa(z) \geqslant \kappa(1) = 2 \sin(1).$$
Hence
\begin{align*}
   \Delta(n;u) & \leqslant \tfrac{1}{2 \sin(1)} \sum_{d \mid n} \kappa \left (\log d - u \right ) \leqslant \tfrac{1}{2 \sin(1)} \int_{-1}^1 \left | \tau(n;v) \right | \, \textrm{d}v \\
   & = \tfrac{1}{\sin(1)} \int_0^1 \left | \tau(n;v) \right | \, \textrm{d}v
\end{align*}
and therefore
$$\sum_{n \leqslant x} t^{\omega(n)}\Delta(n) \leqslant \tfrac{1}{\sin(1)} \int_0^1 \sum_{n \leqslant x} t^{\omega(n)} \left | \tau(n;v) \right | \, \textrm{d}v.$$
Note that the function $n \mapsto t^{\omega(n)} \left| \tau(n;v) \right|$ is multiplicative. Now Lemma~\ref{le:sums_primes} and the condition $t \leqslant 1$ enable us to use Lemma~\ref{le:mean} with $\left( a_1 ; a_2 ; a_3 \right) = \left( 2 + \np{3.88} \times 10^{-8} \, ; \, \np{2.9215} \, ; \, \np{8.17} \right)$, so that we are led to set
$$\lambda_{a,r} := \tfrac{1}{\sin(1)} \, e^{2B + \frac{4a \zeta(r)}{(2 \pi)^r} + \np{2.9215}} (\np{11.17} + \np{3.88} \times 10^{-8}).$$
Thus, Lemmas~\ref{le:mean} and~\ref{le:tau_restreint_bis} then yield, for all $q \geqslant 1$ and $\max \left( e^q,T_1,T_2\right) \leqslant T \leqslant x$,
\begin{align*}
   & \sum_{n \leqslant x} t^{\omega(n)} \Delta(n) \\
   & \leqslant \lambda_{a,r} \, e^{\frac{14a}{(\log T)^r} + \frac{2b}{(\log T)^s}} \, \frac{x}{\log ex} \, \Biggl\{ \int_0^{q/ \log x} \exp \left( 2 t \log \log x  \right) \, \textrm{d}v \\
   & \hspace*{1cm} + \int_{q/\log x}^{q/\log T} \exp\left( \tfrac{4t}{\pi} \log \log x -  \left( 2 - \tfrac{4}{\pi} \right)t \log \left( \tfrac{v}{q}\right) + 4mtq^{-1} \right) \, \textrm{d}v  \\
   & \hspace*{2cm} + \int_{q/\log T}^1 \exp\left( \tfrac{4t}{\pi} \log \log x + \tfrac{4t}{\pi} \log \left( \tfrac{v}{q}\right) + 2t \log \log T  + 4mtq^{-1} \right) \, \textrm{d}v \Biggr\} \\
   & \leqslant \lambda_{a,r} \, e^{\frac{14a}{(\log T)^r} + \frac{2b}{(\log T)^s}} \, x \, \Biggl\{ \frac{\pi e^{4mtq^{-1}}}{4t+\pi} (\log ex)^{-1+4t/\pi} \times \\
   & \hspace*{1cm} \left( q^{-4t/\pi} (\log T)^2 + \frac{2 \pi t q}{(\pi - 2t(\pi-2)) (\log T)^{1+t(4/\pi-2)}} \right)  \\
   & \hspace*{2cm} - \frac{q (\log x)^{2(t-1)} \left( \pi e^{4mtq^{-1}} + 2t(\pi-2) - \pi\right) }{\pi - 2t(\pi-2)} \Biggr\} \\
   & < \frac{\lambda_{a,r} \, \pi e^{4mtq^{-1}}}{4+\pi} \, e^{\frac{14a}{(\log T)^r} + \frac{2b}{(\log T)^s}} \, x \,  (\log ex)^{-1+4t/\pi} \times \\
   & \hspace*{2cm} \left( q^{-4t/\pi} (\log T)^2 + \frac{2 \pi t q}{(\pi - 2t(\pi-2)) (\log T)^{1+t(4/\pi-2)}} \right).
\end{align*}
By Lemma~\ref{le:TNP_explicite}, one may choose $\left( a;b;r;s;T_1;T_2 \right) = \left(\np{4.6};\np{0.6};3;2;2;127 \right)$. The respective choices of $q = \frac{3}{5} \log T$ and $T = e^{8.2} \approx \np{3640.95}$ give
$$\sum_{n \leqslant x} t^{\omega(n)} \Delta(n)  < \np{9380} \, x \,  (\log ex)^{-1+4t/\pi}$$
as required.
\qed

\subsection{Proof of Corollary~\ref{cor:Hooley_restricted_omega}}

Let $t \in \left( 0,1 \right)$ be a parameter at our disposal. Note that, for all $m,n \in \Z_{\geqslant 1}$ and $t \in \left[ 0,1 \right]$, we have
$$t^{\omega(mn)} \Delta (mn) \leqslant t^{\omega(m)} \Delta (m) \underbrace{t^{\omega(n)}}_{\leqslant 1} \tau(n) \leqslant t^{\omega(m)} \Delta (m) \tau(n)$$
so that the Main Theorem~\ref{th:main} applies with $f(n) = t^{\omega(n)} \Delta(n)$ and $k=2$, and along with Proposition~\ref{pro:Delta_bis}, it yields
\begin{align*}
   \sum_{\substack{x-y < x \leqslant x \\ \omega(n) \leqslant j}} \Delta(n) & \leqslant t^{-j} \sum_{x-y < n \leqslant x} t^{\omega(n)} \Delta (n) \\
   & \leqslant \Lambda(2,\ell) \frac{y t^{-j}}{\log x} \; \sum_{n \leqslant x} \frac{t^{\omega(n)}\Delta (n)}{n} \\
   & \leqslant \np{16748} \Lambda (2,\ell) \, y t^{-j} (\log ex)^{-1+4t/\pi}
\end{align*}
and choosing $t = \dfrac{\pi j}{4 \log \log x}$ gives the asserted result. Note that the condition $x > \exp \left( e^{\pi j/4} \right)$ ensures that $0 < t < 1$.
\qed

\subsection*{Acknowledgment}
Many thanks to our \TeX-pert for developing this class file.

\Addresses

\end{document}